\documentclass[12pt,reqno]{amsart}

\usepackage[T1]{fontenc}
\usepackage[utf8]{inputenc}
\usepackage{lmodern}
\usepackage{microtype}
\usepackage{amsmath,amssymb,amsthm,mathtools}
\usepackage{enumitem}
\usepackage[hidelinks]{hyperref}
\usepackage[a4paper,margin=29mm]{geometry}
\newcommand{\aTriv}{\mathrm{aTriv}}
\newcommand{\Triv}{\mathrm{Triv}}

\setlist{nosep,leftmargin=*}

\newtheorem{theorem}{Theorem}[section]
\newtheorem{proposition}[theorem]{Proposition}
\newtheorem{lemma}[theorem]{Lemma}
\newtheorem{corollary}[theorem]{Corollary}
\newtheorem*{theoremA}{Theorem A}
\theoremstyle{definition}

\theoremstyle{remark}

\DeclareMathOperator{\Aut}{Aut}
\DeclareMathOperator{\Inn}{Inn}
\DeclareMathOperator{\Out}{Out}
\DeclareMathOperator{\Img}{Im}

\title[Finite left-simple skew braces]{Classification of finite left-simple skew braces}
\author{Marco Damele}
\address{Dipartimento di Matematica e Informatica, Universit\`a degli Studi di Cagliari, Italy}
\email{marco.damele@unica.it}
\date{}


\begin{document}

\begin{abstract}
We give a complete classification of finite left-simple skew braces. We prove that a non-trivial finite skew brace has no non-trivial proper left ideals if and only if it is either the trivial skew brace of prime order or the almost trivial skew brace associated with a finite non-abelian simple group. 
\end{abstract}

\maketitle

\section{Introduction}

Braces were introduced by Rump \cite{Rump} as an algebraic framework for the study of non-degenerate involutive set-theoretic solutions of the Yang--Baxter equation, and skew braces were subsequently introduced by Guarnieri and Vendramin \cite{GuarnieriVendramin} in order to encompass the non-involutive setting. A skew brace is a triple $(A,\cdot,\circ)$ such that $(A,\cdot)$ and $(A,\circ)$ are groups satisfying
\[
 a\circ(b\cdot c)=(a\circ b)\cdot a^{-1}\cdot(a\circ c)
\]
for all $a,b,c \in A$, where $a^{-1}$ denotes the inverse of $a$ in $(A,\cdot)$. The class of skew braces is considerably broader than the class of groups, since the same underlying set carries two, in general different, group structures. Nevertheless, every group gives rise to two basic examples. If $(G,\cdot)$ is a group, then $(G,\cdot,\cdot)$ is a skew brace, called the \emph{trivial skew brace} on $G$ and denoted by $\Triv(G)$. Likewise, if $\cdot^{\mathrm{opp}}$ denotes the opposite operation on $G$, then $(G,\cdot,\cdot^{\mathrm{opp}})$ is a skew brace, called the \emph{almost trivial skew brace} on $G$ and denoted by $\aTriv(G)$.
 Besides their original connection with the Yang--Baxter equation, skew braces are naturally equivalent to bijective $1$-cocycles and to regular subgroups of holomorphs, and they have close connections with factorisations of groups and Hopf--Galois structures; see, for instance, \cite{GuarnieriVendramin,SmoktunowiczVendramin,StefanelloTrappeniers}.

A natural problem in the theory of finite skew braces is to study how restrictions on their substructures influence the structure of the whole brace. Similar questions are classical in finite group theory, where assumptions on normal, characteristic, maximal or minimal subgroups often lead to strong structural consequences. In the setting of skew braces, one may consider several distinguished classes of substructures, such as sub-skew braces, left ideals and ideals, and investigate finite skew braces in which these substructures are subject to strong restrictions.

Among these, ideals play a particularly important role, since they are precisely the substructures that allow one to define quotient skew braces. Accordingly, a non-trivial skew brace is called \emph{simple} if it has no non-trivial proper ideals. The study of simple braces has been an important part of the subject from its early development. Bachiller constructed the first non-trivial finite simple left braces by means of matched products \cite{BachillerSimple}; asymmetric products and subsequent constructions then produced broad families of finite simple left braces \cite{BCJO,CJO}. In the genuinely skew setting, Byott constructed the first infinite families of finite simple skew braces beyond the classical examples arising from non-abelian simple groups \cite{Byott}. More recently, several works have further developed the theory of finite simple skew braces from both the structural and constructive points of view; see, for instance, \cite{DameleCyclicSylow,DameleErcan,DameleKernelWreath,DameleInverseBaer}. These results show that finite simple skew braces can exhibit a remarkably rich range of behaviour, and make stronger restrictions on distinguished substructures a natural source of rigidity.

A particularly strong restriction was considered by Ballester-Bolinches et al.~\cite{BallesterAdv}. In their study of solvability and substructure in skew braces, they proved that a non-trivial finite skew brace with no non-trivial proper sub-skew braces is necessarily the trivial skew brace of prime order. This may be viewed as a skew-brace analogue of the classical fact that a finite group with no non-trivial proper subgroups is cyclic of prime order. In the skew-brace setting, however, the interaction between the two group structures makes this rigidity phenomenon substantially less immediate.

The notion of left-simplicity was recently studied systematically by Tsang \cite{Tsang}. Recall that, for $a,x \in A$,
\[
 \lambda_a(x)=a^{-1}\cdot(a\circ x)
\]
defines an automorphism of $(A,\cdot)$ and the map $\lambda: a\mapsto\lambda_a$ is a homomorphism from $(A,\circ)$ to $\Aut(A,\cdot)$. This is called the lambda-map of $A$. A subgroup $I\le (A,\cdot)$ is a \emph{left ideal} if it is invariant under every $\lambda_a$. A non-trivial skew brace $A$ is called \emph{left-simple} if its only left ideals are $1$ and $A$. Since every characteristic subgroup of $(A,\cdot)$ is a left ideal, the additive group of a finite left-simple skew brace is characteristically simple. Hence $ (A,\cdot)\cong T^n$
for some finite simple group $T$. Tsang initiated the classification of finite left-simple skew braces. She completely settled two fundamental cases. If $T$ is cyclic of prime order, then left-simplicity forces $n=1$, and one obtains the trivial skew brace of prime order. If $T$ is non-abelian and $n=1$, then a finite skew brace with additive group $T$ is left-simple precisely when it is almost trivial.
For $T$ non-abelian simple and $n\ge2$, Tsang obtained strong necessary conditions involving the two canonical actions, their images in $\Aut(T^n)$ and the induced transitive action on the simple direct factors. The general existence problem for additive group $T^n$, $n\ge2$, remained open.

Our main theorem eliminates this remaining case and yields a complete classification.

\begin{theoremA}
Let $A$ be a finite skew brace. Then $A$ is left-simple if and only if exactly one of the following holds:
\begin{enumerate}[label=(\roman*)]
\item $A \simeq \Triv(C_p)$ for some prime $p$;
\item $A \simeq \aTriv(S)$ for a finite non-abelian simple group $S$.
\end{enumerate}
\end{theoremA}

Theorem~A completes the classification initiated by Tsang and shows that left-simplicity is a very rigid condition. It also recovers, as an immediate corollary, the theorem of Ballester-Bolinches et al.~\cite{BallesterAdv} on finite skew braces with no non-trivial proper sub-skew braces. Finally, through the correspondence between left ideals and Hopf subalgebras \cite{StefanelloTrappeniers,Tsang}, Theorem~A yields a complete classification of minimal Hopf--Galois structures on finite Galois extensions.

The paper is deliberately short. Section~2 collects the consequences of Tsang's work together with the finite-group facts needed later. Section~3 proves Theorem~A and records the consequences for sub-skew-brace-free skew braces and minimal Hopf--Galois structures.

\section{Preliminary lemmas}

 Let $(A,\cdot,\circ)$ be a finite skew brace as in the introduction and let $\lambda$ be its lambda-map. Define for $a,x \in A$ $ \rho_a(x)=(a\circ x)\cdot a^{-1}$.
Then $ \rho:(A,\circ)\longrightarrow\Aut(A,\cdot)$ sending $a \in A$ to $\rho(a)=\rho_{a}$
is a group homomorphism. If $\operatorname{conj}(a)$ denotes conjugation by $a$ in $(A,\cdot)$, then
\begin{equation}\label{eq:lambda-rho}
 \rho_a=\operatorname{conj}(a)\lambda_a,
 \qquad
 \lambda_a=\operatorname{conj}(a^{-1})\rho_a.
\end{equation}
Consequently
\begin{equation}\label{eq:HNKN}
 \Img(\lambda)\Inn(A,\cdot)=\Img(\rho)\Inn(A,\cdot).
\end{equation}
This is the subgroup denoted $\Gamma$ in \cite[Proposition~2.1]{Tsang}. Let 
\begin{equation}
    \operatorname{conj}: (A,\cdot) \longrightarrow \Aut(A,\cdot)
\end{equation}
be the conjugation map, namely the map sending $a \in A$ to $\operatorname{conj}(a)$.
Tsang considers the two canonical subgroups $ J_1:=\operatorname{conj}^{-1}(\Img(\lambda))$ and $ J_2:=\ker(\rho),$
and proves that both are left ideals \cite[Proposition~2.2]{Tsang}. We isolate the consequences that will be used in the proof.

\begin{lemma}\label{lem:TsangSetup}
Let $(A,\cdot,\circ)$ be a left-simple skew brace and suppose that $ (A,\cdot) = S^n,$
where $S$ is non-abelian simple and $n\ge2$. Then:
\begin{enumerate}[label=(\alph*)]
\item $\Img(\lambda)\cap\Inn(A,\cdot)=1$;
\item $\rho$ is injective, and hence $|\Img(\rho)|=|A|=|S|^n=|\Inn(A,\cdot)|$;
\item $\Img(\lambda)\Inn(A,\cdot)=\Img(\rho)\Inn(A,\cdot)$;
\item If for
$\alpha\in\Aut(S^n)$ we denote by $\sigma_\alpha\in S_n$ the permutation
defined by $\alpha(S_i)=S_{\sigma_\alpha(i)}$ for every $i$,
then $\{\sigma_h:h\in\Img(\lambda)\}
=
\{\sigma_k:k\in\Img(\rho)\},$
and this common subgroup of $S_n$ acts transitively on
$\{1,\ldots,n\}$.
\end{enumerate}
\end{lemma}

\begin{proof}
The argument is essentially contained in \cite[Section~3.2]{Tsang}. We recall it in some detail in order to make the proof of Theorem~A self-contained. We first prove that $\Img(\lambda)\cap\Inn(A,\cdot)=1.$ Write $(A,\cdot)=S_1\times\cdots\times S_n$, where $S_i = S$ for every $i\in\{1,\ldots,n\}$. Suppose first that $\Img(\lambda)\le\Inn(A,\cdot)$. Then every element of $\Img(\lambda)$ is an inner automorphism of $(A,\cdot)$ and therefore preserves each direct factor $S_i$. Hence $\lambda_a(S_i)=S_i$ for every $a\in A$ and every $i$, so each $S_i$ is a left ideal of $A$. Since $n\ge2$, each $S_i$ is non-trivial and proper in $(A,\cdot)$, contradicting the left-simplicity of $A$. Therefore $\Img(\lambda)\not\le\Inn(A,\cdot)$.
Now let $J_1=\operatorname{conj}^{-1}(\Img(\lambda)).$
Since \(J_1\) is a left ideal and \(A\) is left-simple, we have $J_1=1$ or $J_1=A$.
Assume that \(J_1=A\). Then $\Inn(A,\cdot)=\operatorname{conj}(A)\le\Img(\lambda).$
Moreover, since $(A,\cdot)=S^n$ is centreless, the conjugation map identifies $(A,\cdot)$ with $\Inn(A,\cdot)$, and hence we have
$|\Img(\lambda)|\le |(A,\circ)|=|A|=|\Inn(A,\cdot)|,$
and therefore \(\Inn(A,\cdot)\le\Img(\lambda)\) forces \(\Img(\lambda)=\Inn(A,\cdot)\), a contradiction.
Consequently $J_1=1.$ Thus $\operatorname{conj}^{-1}(\Img(\lambda)\cap\Inn(A,\cdot))
=
\operatorname{conj}^{-1}(\Img(\lambda))
=
J_1=1.$ Since \(\operatorname{conj}\colon A\to\Inn(A,\cdot)\) is an isomorphism, we have $\Img(\lambda)\cap\Inn(A,\cdot)=1.$ We next prove that \(\rho\) is injective. Recall that $J_2=\ker(\rho).$
Again \(J_2\) is a left ideal of \(A\), so left-simplicity gives $J_2=1$ or $J_2=A.$
If \(J_2=A\), then \(\rho\) is trivial and hence
$\Img(\rho)=1.$
Using the relation between the two canonical actions we would obtain, for every $a \in A$
$\lambda_a=\operatorname{conj}(a^{-1})$.
Thus $\Img(\lambda)=\Inn(A,\cdot),$ a contradiction. Therefore $J_2=1,$
and hence $\rho$ is injective. In particular we have $|\Img(\rho)|=|A|=|\Inn(A,\cdot)|.$
For every $\alpha\in\Aut(S^n),$
there exists a unique permutation $\sigma_\alpha\in S_n$ such that $\alpha(S_i)=S_{\sigma_\alpha(i)}$ for every $i$. Consider the map
\[
\Phi:\Aut(S^n)\longrightarrow S_n,
\qquad
\alpha\longmapsto \sigma_\alpha.
\]
This is a group homomorphism and $\Inn(A,\cdot)\le\ker\Phi.$
Applying $\Phi$ to the equality $\Img(\lambda)\Inn(A,\cdot)=\Img(\rho)\Inn(A,\cdot)$ and using $\Phi(\Inn(A,\cdot))=1$, we obtain $\Phi(\Img(\lambda))=\Phi(\Img(\rho)).$
Thus
\[
\{\sigma_h:h\in\Img(\lambda)\}
=
\{\sigma_k:k\in\Img(\rho)\}
\le S_n.
\]
We claim that this common subgroup acts transitively on $\{1,\ldots,n\}.$
Suppose otherwise. Then it has a non-empty proper orbit $\varnothing\neq\Omega\subsetneq\{1,\ldots,n\}.$
Set $S_\Omega:=\prod_{i\in\Omega}S_i.$
Let $h\in\Img(\lambda)$, and let $\sigma_h$ be the permutation induced by $h$. We have
\[
h(S_\Omega)
=
h\left(\prod_{i\in\Omega}S_i\right)
=
\prod_{i\in\Omega}S_{\sigma_h(i)}
=
S_\Omega.
\]
This means that $\lambda_a(S_\Omega)=S_\Omega$ for every $a \in A$.
Thus $S_\Omega$ is a left ideal of $A$. Since $\Omega$ is non-empty and
proper we have $1<S_\Omega<(A,\cdot),$
contradicting the left-simplicity of $A$. Hence the common subgroup acts transitively on
$\{1,\ldots,n\}$.

\end{proof}

For a finite non-abelian simple group $S$, let
\[
 m(S):=\min\{[S:M]:M<S\},
 \qquad q(S):=|\Out(S)|.
\]
The following estimates are consequences of the classification of finite simple groups.

\begin{lemma}\label{lem:out-bounds}
Let $S$ be a finite non-abelian simple group. Then $ q(S)<m(S)$ and $ q(S)^3<|S|.$
\end{lemma}

\begin{proof}
For a non-abelian simple group, the least degree $P(S)$ of a faithful permutation representation equals $m(S)$. Holt and Tracey prove that $2P(S)>3|\Out(S)|$ for every finite non-abelian simple group \cite[Lemma~8]{HoltTracey}, and hence $q(S)<m(S)$. The estimate $q(S)^3<|S|$ is \cite[Lemma~4.8]{Fawcett}.
\end{proof}

We recall that a subgroup $L\le G_1\times\cdots\times G_n$ is called a \emph{subdirect subgroup} if each coordinate projection $\pi_i:L\longrightarrow G_i$
is surjective. Let $S_1,\ldots,S_n$ be isomorphic non-abelian simple groups, and let $\Omega\subseteq\{1,\ldots,n\}$. A subgroup $D\le\prod_{i\in\Omega}S_i$ is called a \emph{full diagonal subgroup} if, for every $i\in\Omega$, the restriction to $D$ of the coordinate projection $\pi_i:\prod_{j\in\Omega}S_j\to S_i$ is an isomorphism. Equivalently, fixing $i_0\in\Omega$, there exist isomorphisms $\varphi_i:S_{i_0}\to S_i$ for every $i\in\Omega$, with $\varphi_{i_0}=\operatorname{id}$, such that
\[
D=\{(\varphi_i(s))_{i\in\Omega}:s\in S_{i_0}\}.
\]

\begin{lemma}\label{lem:Scott}
Let $ N=S_1\times\cdots\times S_n,$
where the $S_i$ are isomorphic non-abelian simple groups. If $L\le N$ is subdirect, then there is a partition $ \{1,\ldots,n\}=\Omega_1\sqcup\cdots\sqcup\Omega_r$
such that $ L=D_1\times\cdots\times D_r,$
where each $D_j\cong S$ is a full diagonal subgroup of $\prod_{i\in\Omega_j}S_i$.
\end{lemma}

\begin{proof}
See \cite[Theorem~4.16]{PraegerSchneider}.
\end{proof}

Finally, for a prime $p$ and a positive integer $m$, let $v_p(m)$ denote the $p$-adic valuation of $m$, namely the largest integer $a\ge0$ such that $p^a\mid m$. We shall use the standard additivity property
\[
v_p(ab)=v_p(a)+v_p(b)
\]
for positive integers $a$ and $b$; see, for instance, \cite[Chapter~1]{NivenZuckermanMontgomery}. For a real number $x$, we denote by $\lfloor x\rfloor$ the greatest integer less than or equal to $x$.

\begin{lemma}\label{lem:vp-bound}
Let $p$ be a prime and let $n\ge1$. Then
\[
v_p(n!)<\frac{n}{p-1}.
\]
In particular, $v_p(n!)<n$.
\end{lemma}

\begin{proof}
By Legendre's formula \cite[Theorem~6.5.1]{AndreescuAndrica},
\[
v_p(n!)=\sum_{j\ge1}\left\lfloor\frac{n}{p^j}\right\rfloor.
\]
Choose $J$ such that $p^J>n$. Then
$\lfloor n/p^j\rfloor=0$ for every $j\ge J$, and hence
\[
v_p(n!)
=
\sum_{j=1}^{J-1}\left\lfloor\frac{n}{p^j}\right\rfloor
\le
\sum_{j=1}^{J-1}\frac{n}{p^j}
<
\sum_{j=1}^{\infty}\frac{n}{p^j}
=
\frac{n}{p-1}.
\]
Thus $v_p(n!)<n/(p-1)$. In particular, $v_p(n!)<n$.
\end{proof}

\section{Proof of Theorem A}

The only case not covered by Tsang is the proper direct power of a non-abelian simple group. We first exclude it.

\begin{proposition}\label{prop:no-power}
Let $S$ be a finite non-abelian simple group and let $n\ge2$. There is no finite left-simple skew brace $(A,\cdot,\circ)$ such that $ (A,\cdot) = S^n.$
\end{proposition}

\begin{proof}

Assume, for a contradiction, that such a skew brace $(A,\cdot,\circ)$ exists. Write $(A,\cdot)=S_1\times\cdots\times S_n$, where $S_i=S$ for every $i$. We will work with a particular subgroup inside $(A,\cdot)$. Let $\operatorname{conj}:(A,\cdot)\to\Aut(A,\cdot)$ be the conjugation map and set
\[
M:=\operatorname{conj}^{-1}(\Img(\rho)).
\]
Since $\Img(\rho)$ is a subgroup of $\Aut(A,\cdot)$ and $\operatorname{conj}$ is a homomorphism, $M$ is a subgroup of $(A,\cdot)$.
\medskip
\noindent\emph{Step 1: $k(M)=M$ for every $k\in\Img(\rho)$}
\medskip 

 We claim that $M$ is invariant under every element of $\Img(\rho)$. Let $k\in\Img(\rho)$ and $x\in M$. Since $\operatorname{conj}(x)\in\Img(\rho)$, we have $k\operatorname{conj}(x)k^{-1}\in\Img(\rho)$. Moreover $k\operatorname{conj}(x)k^{-1}=\operatorname{conj}(k(x))$, and hence $k(x)\in M$. Therefore $k(M)=M$ for every $k\in\Img(\rho)$. For $1\le i\le n$, let $\pi_i:(A,\cdot)\to S_i$ be the $i$-th coordinate projection and put $M_i:=\pi_i(M)$.

\medskip 
\noindent\emph{Step 2: $|M_1|=\cdots=|M_n|$}
\medskip 

 By Lemma~\ref{lem:TsangSetup}, the subgroup $\{\sigma_k:k\in\Img(\rho)\}\le S_n$ acts transitively on $\{1,\ldots,n\}$, where $\sigma_k$ is defined by $k(S_i)=S_{\sigma_k(i)}$. Let $i,j\in\{1,\ldots,n\}$. By transitivity, there exists $k\in\Img(\rho)$ such that $\sigma_k(i)=j$. Let $k_i:S_i\to S_j$ be the restriction of $k$ to $S_i$; this is an isomorphism. Since $k$ permutes the direct factors according to $\sigma_k$, we have $\pi_j(k(m))=k_i(\pi_i(m))$ for every $m\in(A,\cdot)$. Hence, since $k(M)=M$, if $x\in M_i=\pi_i(M)$ and $m\in M$ satisfies $\pi_i(m)=x$, then $k(m)\in M$ and $\pi_j(k(m))=k_i(x)$. Hence $k_i(M_i)\le M_j$. Applying the same argument to $k^{-1}$ gives $M_j\le k_i(M_i)$, and therefore $k_i(M_i)=M_j$. Thus $M_i\cong M_j$. Since $i$ and $j$ are arbitrary, it follows that $|M_1|=\cdots=|M_n|$.

\medskip

\noindent\emph{Step 3: An arithmetic restriction on $[(A,\cdot):M]$.}

\medskip

By Lemma~\ref{lem:TsangSetup},
\begin{align*}
|\Img(\lambda)||\Inn(A,\cdot)|
&=|\Img(\lambda)\Inn(A,\cdot)|
&&\bigl(\Img(\lambda)\cap\Inn(A,\cdot)=1\bigr)\\
&=|\Img(\rho)\Inn(A,\cdot)|
&&\bigl(\Img(\lambda)\Inn(A,\cdot)=\Img(\rho)\Inn(A,\cdot)\bigr)\\
&=\frac{|\Img(\rho)||\Inn(A,\cdot)|}{|\Img(\rho)\cap\Inn(A,\cdot)|}
&&\bigl(\text{product formula}\bigr)\\
&=\frac{|\Inn(A,\cdot)|^2}{|\Img(\rho)\cap\Inn(A,\cdot)|}
&&\bigl(|\Img(\rho)|=|\Inn(A,\cdot)|\bigr).
\end{align*}
Therefore
$|\Img(\lambda)|=[\Inn(A,\cdot):\Img(\rho)\cap\Inn(A,\cdot)]$.
Since $(A,\cdot)$ is centreless, the conjugation map
$\operatorname{conj}:(A,\cdot)\to\Inn(A,\cdot)$ is an isomorphism. Thus we have $\operatorname{conj}(M)=\Img(\rho)\cap\Inn(A,\cdot).$
In particular we get $[(A,\cdot):M]
=
[\Inn(A,\cdot):\Img(\rho)\cap\Inn(A,\cdot)].$
Therefore
\begin{equation}\label{eq:index-M}
[(A,\cdot):M]=|\Img(\lambda)|.
\end{equation}

Consider the natural quotient homomorphism $\Aut(A,\cdot)\to\Aut(A,\cdot)/\Inn(A,\cdot)$. Its restriction to $\Img(\lambda)$ has kernel $\Img(\lambda)\cap\Inn(A,\cdot)=1$, so $\Img(\lambda)$ embeds into $\Aut(A,\cdot)/\Inn(A,\cdot)$. Writing $q:=|\Out(S)|$, by \cite[Theorem~3.3.20]{Robinson} we have
\[
\Aut(A,\cdot)/\Inn(A,\cdot)
\cong
\Out(S)^n\rtimes S_n.
\]
Consequently
\[
|\Aut(A,\cdot)/\Inn(A,\cdot)|=q^n n!.
\] Hence, by Lagrange's theorem and \eqref{eq:index-M},
\begin{equation}\label{eq:divisibility}
[(A,\cdot):M]=|\Img(\lambda)|\mid q^n n!.
\end{equation}
Since $|M_1|=\cdots=|M_n|$ and $|S_1|=\cdots=|S_n|=|S|$, either $M_i<S_i$ for every $i$ or $M_i=S_i$ for every $i$.

\medskip
\noindent\emph{Step 4: The case of proper coordinate projections.}
\medskip

Suppose that $M_i<S_i$ for every $i$. Since all the $M_i$ have the same order, the indices $[S_i:M_i]$ are equal; put $d:=[S_i:M_i]$. Since $M\le M_1\times\cdots\times M_n$, we have
$d^n=[(A,\cdot):M_1\times\cdots\times M_n]\mid[(A,\cdot):M].$
Together with \eqref{eq:divisibility}, this gives
\begin{equation}\label{eq:d-div}
d^n\mid q^n n!.
\end{equation}
We claim that $d\mid q$. Let $p$ be any prime and put $e:=v_p(d)$ and $b:=v_p(q)$. From \eqref{eq:d-div} we have $v_p(d^n)\le v_p(q^n n!).$ Since $v_p(d^n)=ne$ and $v_p(q^n n!)=nb+v_p(n!)$, we obtain $ne\le nb+v_p(n!)$, hence $n(e-b)\le v_p(n!)$. If $e>b$, then $e-b\ge1$, so $n\le n(e-b)\le v_p(n!)$, contradicting $v_p(n!)<n$. Thus $e\le b$ for every prime $p$, and therefore $d\mid q$. On the other hand, since $M_i<S_i\cong S$, we have $d\ge m(S)>q$ by Lemma~\ref{lem:out-bounds}, a contradiction.

\medskip 
\noindent\emph{Step 5: The case of surjective coordinate projections.}
\medskip

Suppose that $M_i=S_i$ for every $i$. Since $M_i=\pi_i(M)$, every coordinate projection $\pi_i:M\to S_i$ is surjective. Thus $M$ is a subdirect subgroup of $(A,\cdot)=S_1\times\cdots\times S_n$. By Lemma~\ref{lem:Scott}, there exists a partition
$\{1,\ldots,n\}=\Omega_1\sqcup\cdots\sqcup\Omega_r$
and, for every $j\in\{1,\ldots,r\}$, a full diagonal subgroup
$D_j\le\prod_{i\in\Omega_j}S_i$ such that
$M=D_1\times\cdots\times D_r$. Since each $D_j$ projects isomorphically onto every factor $S_i$ with $i\in\Omega_j$, we have $D_j\cong S$. Therefore
\begin{equation}\label{eq:M-order}
|M|=|D_1|\cdots|D_r|=|S|^r.
\end{equation}
For each $j$, the set $\Omega_j$ is precisely the set of coordinates on which $D_j$ has non-trivial projection; we call it the support of $D_j$. Since $M=D_1\times\cdots\times D_r$ and every $D_j$ is non-abelian simple, the subgroups $D_1,\ldots,D_r$ are exactly the minimal normal subgroups of $M$. Let $k\in\Img(\rho)$. Since $k(M)=M$, the restriction $k|_M$ is an automorphism of $M$, and therefore permutes the minimal normal subgroups $D_1,\ldots,D_r$. Thus, for every $j$, there exists some $\ell$ such that $k(D_j)=D_\ell$. Let $\sigma_k$ be the permutation induced by $k$ on the factors $S_1,\ldots,S_n$, so that $k(S_i)=S_{\sigma_k(i)}$. If $i\in\Omega_j$, then $D_j$ has non-trivial projection onto $S_i$, and therefore $k(D_j)$ has non-trivial projection onto $S_{\sigma_k(i)}$. Hence the support of $k(D_j)$ is $\sigma_k(\Omega_j)$. Since $k(D_j)=D_\ell$, whose support is $\Omega_\ell$, we obtain
$\sigma_k(\Omega_j)=\Omega_\ell.$
Therefore $\{\sigma_k:k\in\Img(\rho)\}$ permutes the sets $\Omega_1,\ldots,\Omega_r$.
We now show that all the $\Omega_j$ have the same cardinality. Let $\Omega_a$ and $\Omega_b$ be two of these sets, and choose $i\in\Omega_a$ and $j\in\Omega_b$. Since $\{\sigma_k:k\in\Img(\rho)\}$ acts transitively on $\{1,\ldots,n\}$, there exists $\sigma\in\{\sigma_k:k\in\Img(\rho)\}$ such that $\sigma(i)=j$. Since this group permutes the sets $\Omega_1,\ldots,\Omega_r$, the set $\sigma(\Omega_a)$ is one of the $\Omega_\ell$. Moreover, it contains $\sigma(i)=j$, and the sets $\Omega_1,\ldots,\Omega_r$ are pairwise disjoint. Hence $\sigma(\Omega_a)=\Omega_b$. Therefore $|\Omega_a|=|\Omega_b|$. Since $a$ and $b$ are arbitrary, all the supports have the same cardinality; write $|\Omega_j|=c$. Since they form a partition of $\{1,\ldots,n\}$, we have $n=rc$.

Suppose that $c=1$. Then every $\Omega_j$ consists of a single index, so every full diagonal subgroup $D_j$ is simply one of the factors $S_i$. Hence $M=S_1\times\cdots\times S_n=(A,\cdot)$, and therefore $[(A,\cdot):M]=1$. By \eqref{eq:index-M}, $|\Img(\lambda)|=1$. Thus $\{\sigma_h:h\in\Img(\lambda)\}=1$, and by Lemma~\ref{lem:TsangSetup} also $\{\sigma_k:k\in\Img(\rho)\}=1$, contradicting its transitivity on $\{1,\ldots,n\}$ because $n\ge2$. Therefore $c\ge2$, and hence
\begin{equation}\label{eq:r-bound}
r\le\frac n2.
\end{equation}
By \eqref{eq:M-order},
$[(A,\cdot):M]=|S|^{n-r}$.
Consequently \eqref{eq:divisibility} yields
\begin{equation}\label{eq:S-div}
|S|^{n-r}\mid q^n n!.
\end{equation}
By Lemma~\ref{lem:out-bounds}, $q^3<|S|$. Hence there exists a prime $p\mid |S|$ such that, setting $a:=v_p(|S|)$ and $b:=v_p(q)$, one has
\begin{equation}\label{eq:a>3b}
a>3b.
\end{equation}
Indeed, if $v_p(|S|)\le3v_p(q)$ for every prime $p\mid |S|$, then $v_p(|S|)\le v_p(q^3)$ for every prime $p$, so $|S|\mid q^3$, contradicting $q^3<|S|$. Fix such a prime $p$. From \eqref{eq:S-div} we have $v_p(|S|^{n-r})\le v_p(q^n n!)$. Since $v_p(|S|^{n-r})=(n-r)a$ and $v_p(q^n n!)=nb+v_p(n!)$, we obtain $(n-r)a\le nb+v_p(n!)$. By \eqref{eq:r-bound}, $n-r\ge n/2$, and therefore
\[
\frac{na}{2}\le(n-r)a\le nb+v_p(n!).
\]
Hence
\begin{equation}\label{eq:valuation-final}
n\left(\frac a2-b\right)\le v_p(n!).
\end{equation}
Suppose first that $b\ge1$. Since $a>3b$ and $a,b$ are integers, $a\ge3b+1$, so
\[
\frac a2-b\ge\frac{3b+1}{2}-b=\frac{b+1}{2}\ge1.
\]
By \eqref{eq:valuation-final}, $n\le v_p(n!)$, contradicting $v_p(n!)<n$. Suppose now that $b=0$ and $p\ge3$. Since $p\mid|S|$, we have $a\ge1$, and \eqref{eq:valuation-final} gives $n/2\le v_p(n!)$. On the other hand,
\[
v_p(n!)<\frac{n}{p-1}\le\frac n2,
\]
a contradiction. Finally, suppose that $b=0$ and $p=2$. Since the prime $p$ was chosen to divide $|S|$, we have $2\mid|S|$. In fact, $4\mid|S|$. Indeed, if $v_2(|S|)=1$ and $P$ is a Sylow $2$-subgroup of $S$, then $|P|=2$. Since $\Aut(P)=1$, we have $N_S(P)=C_S(P)$. Burnside's normal $p$-complement theorem \cite{Robinson} then gives a normal $2$-complement in $S$, contradicting simplicity. Hence $a=v_2(|S|)\ge2$. Thus $a/2-b=a/2\ge1$, and \eqref{eq:valuation-final} gives
$n\le v_2(n!)<n,$
again a contradiction. Thus every case leads to a contradiction, proving the proposition.

\end{proof}

\begin{proof}[Proof of Theorem A]
Suppose first that $A$ is finite and left-simple. Every characteristic subgroup of $(A,\cdot)$ is a left ideal, so $(A,\cdot)$ is characteristically simple. Hence we can assume that $ (A,\cdot) = T^n$
for a finite simple group $T$ and an integer $n\ge1$. If $T$ is abelian, then $T=C_p$ for some prime $p$, and \cite[Corollary~1.3]{Tsang} gives $A\simeq\Triv(C_p)$ directly. Assume that $T$ is non-abelian. Proposition~\ref{prop:no-power} excludes $n\ge2$. Thus $n=1$, and \cite[Theorem~1.4(a)]{Tsang} shows that $A$ is almost trivial. Conversely, the trivial skew brace on $C_p$ is left-simple because $C_p$ has no non-trivial proper subgroup. If $S$ is non-abelian simple and $A$ is almost trivial, then
$ \lambda_a(x)=a^{-1}xa,$
so the left ideals are exactly the normal subgroups of $S$. Hence $A$ is left-simple.
\end{proof}

The first consequence is the result of Ballester-Bolinches et al.~\cite{BallesterAdv}.

\begin{corollary}\label{cor:subbracefree}
Let $A$ be a non-trivial finite skew brace with no non-trivial proper sub-skew braces. Then $A$ is the trivial skew brace of prime order.
\end{corollary}

\begin{proof}
Every left ideal is a sub-skew brace, so $A$ is left-simple. By Theorem~A, either $A$ is trivial of prime order or it is almost trivial on a finite non-abelian simple group $S$. The second possibility is impossible because every finite non-abelian simple group has non-trivial proper subgroups, and in an almost trivial skew brace every subgroup is a sub-skew brace.
\end{proof}

Finally, let $L/K$ be a finite Galois extension. A Hopf--Galois structure $\mathcal H$ on $L/K$ is called \emph{minimal} if $\dim_K(\mathcal H)\neq1$ and its only Hopf $K$-subalgebras are $K$ and $\mathcal H$. Under the bijective skew-brace/Hopf--Galois correspondence used in \cite[Theorem~4.2]{Tsang}, Hopf--Galois structures on $L/K$ correspond to skew brace structures on $\operatorname{Gal}(L/K)$, and \cite[Theorem~4.3]{Tsang} identifies minimal Hopf--Galois structures precisely with left-simple skew braces. Under the same correspondence, the trivial and almost trivial skew braces correspond, respectively, to the classical and canonical non-classical structures.

\begin{corollary}\label{cor:HG}
Let $L/K$ be a finite Galois extension and let $\mathcal H$ be a Hopf--Galois structure on $L/K$. Then $\mathcal H$ is minimal if and only if one of the following holds:
\begin{enumerate}[label=(\roman*)]
\item $\operatorname{Gal}(L/K)\cong C_p$ for some prime $p$ and $\mathcal H$ is the classical structure;
\item $\operatorname{Gal}(L/K)$ is finite non-abelian simple and $\mathcal H$ is the canonical non-classical structure.
\end{enumerate}
\end{corollary}

\begin{proof}
By \cite[Theorem~4.3]{Tsang}, $\mathcal H$ is minimal if and only if its associated skew brace is left-simple. Theorem~A gives exactly the two possibilities in the statement, and the identifications with the classical and canonical non-classical structures follow from the correspondence recalled in \cite[Theorem~4.2]{Tsang}.
\end{proof}


\begin{thebibliography}{99}\small

\bibitem{BachillerSimple}
D.~Bachiller,
\emph{Extensions, matched products, and simple braces},
J. Pure Appl. Algebra \textbf{222} (2018), no.~7, 1670--1691.
\href{https://doi.org/10.1016/j.jpaa.2017.07.017}{doi:10.1016/j.jpaa.2017.07.017}.

\bibitem{BCJO}
D.~Bachiller, F.~Ced\'o, E.~Jespers and J.~Okni\'nski,
\emph{Asymmetric product of left braces and simplicity; new solutions of the Yang--Baxter equation},
Commun. Contemp. Math. \textbf{21} (2019), no.~8, 1850042.
\href{https://doi.org/10.1142/S0219199718500426}{doi:10.1142/S0219199718500426}.

\bibitem{BallesterAdv}
A.~Ballester-Bolinches, R.~Esteban-Romero, P.~Jim\'enez-Seral and V.~P\'erez-Calabuig,
\emph{Soluble skew left braces and soluble solutions of the Yang--Baxter equation},
Adv. Math. \textbf{455} (2024), 109880.
\href{https://doi.org/10.1016/j.aim.2024.109880}{doi:10.1016/j.aim.2024.109880}.

\bibitem{Byott}
N.~P. Byott,
\emph{On a family of simple skew braces},
J. Algebra Appl., to appear.
\href{https://doi.org/10.1142/S0219498826502300}{doi:10.1142/S0219498826502300}.

\bibitem{CJO}
F.~Ced\'o, E.~Jespers and J.~Okni\'nski,
\emph{An abundance of simple left braces with abelian multiplicative Sylow subgroups},
Rev. Mat. Iberoam. \textbf{36} (2020), no.~5, 1309--1332.
\href{https://doi.org/10.4171/RMI/1168}{doi:10.4171/RMI/1168}.


\bibitem{DameleCyclicSylow}
M.~Damele,
\emph{Simple skew braces with cyclic Sylow subgroups},
preprint, 2026, arXiv:2607.17125.

\bibitem{DameleErcan}
M.~Damele and G.~Ercan,
\emph{Ideals and solvability in skew braces},
preprint, 2026, arXiv:2607.19955.

\bibitem{DameleInverseBaer}
M.~Damele,
\emph{Inverse Baer deformations and finite simple skew braces of nilpotent type},
preprint, 2026, arXiv:2609.24640.

\bibitem{DameleKernelWreath}
M.~Damele,
\emph{Kernel--wreath constructions and infinite families of finite simple skew braces},
preprint, 2026, arXiv:2609.20401.

\bibitem{NivenZuckermanMontgomery}
I.~Niven, H.~S.~Zuckerman and H.~L.~Montgomery,
\emph{An Introduction to the Theory of Numbers},
5th ed., John Wiley \& Sons, New York, 1991.

\bibitem{AndreescuAndrica}
T.~Andreescu and D.~Andrica,
\emph{Number Theory: Structures, Examples, and Problems},
Birkh\"auser, Boston, 2009.
\href{https://doi.org/10.1007/b11856}{doi:10.1007/b11856}.

\bibitem{Fawcett}
J.~B. Fawcett,
\emph{The base size of a primitive diagonal group},
J. Algebra \textbf{375} (2013), 302--321.
\href{https://doi.org/10.1016/j.jalgebra.2012.11.020}{doi:10.1016/j.jalgebra.2012.11.020}.

\bibitem{GuarnieriVendramin}
L.~Guarnieri and L.~Vendramin,
\emph{Skew braces and the Yang--Baxter equation},
Math. Comp. \textbf{86} (2017), no.~307, 2519--2534.
\href{https://doi.org/10.1090/mcom/3161}{doi:10.1090/mcom/3161}.

\bibitem{HoltTracey}
D.~F. Holt and G.~Tracey,
\emph{Minimal sized generating sets of permutation groups},
J. Algebra \textbf{703} (2026), 128--139.
\href{https://doi.org/10.1016/j.jalgebra.2025.09.001}{doi:10.1016/j.jalgebra.2025.09.001}.


\bibitem{PraegerSchneider}
C.~E. Praeger and C.~Schneider,
\emph{Permutation Groups and Cartesian Decompositions},
London Math. Soc. Lecture Note Ser., vol.~449, Cambridge University Press, Cambridge, 2018.

\bibitem{Rump}
W.~Rump,
\emph{Braces, radical rings, and the quantum Yang--Baxter equation},
J. Algebra \textbf{307} (2007), no.~1, 153--170.
\href{https://doi.org/10.1016/j.jalgebra.2006.03.040}{doi:10.1016/j.jalgebra.2006.03.040}.

\bibitem{SmoktunowiczVendramin}
A.~Smoktunowicz and L.~Vendramin,
\emph{On skew braces},
J. Combin. Algebra \textbf{2} (2018), no.~1, 47--86.
\href{https://doi.org/10.4171/JCA/2-1-3}{doi:10.4171/JCA/2-1-3}.

\bibitem{StefanelloTrappeniers}
L.~Stefanello and S.~Trappeniers,
\emph{On the connection between Hopf--Galois structures and skew braces},
Bull. Lond. Math. Soc. \textbf{55} (2023), no.~4, 1726--1748.
\href{https://doi.org/10.1112/blms.12815}{doi:10.1112/blms.12815}.

\bibitem{Tsang}
C.~(S.~Y.) Tsang,
\emph{Skew braces with no proper left ideals},
J. Non-Assoc. Struct. \textbf{1} (2026), no.~1, Paper No.~5.
\href{https://doi.org/10.46298/jonas.17536}{doi:10.46298/jonas.17536}.

\bibitem{Robinson}
D.~J.~S. Robinson,
\emph{A Course in the Theory of Groups},
2nd ed., Graduate Texts in Mathematics, vol.~80,
Springer-Verlag, New York, 1996.

\end{thebibliography}
\end{document}